\documentclass[10pt,oneside,a4paper]{article}
\usepackage{titlesec}
\usepackage{authblk}
\titleformat{\subsection}[runin]{\normalfont\bfseries}{\thesubsection.}{.5em}{}[.~ ]
\titlespacing{\subsection}{0pt}{1.5ex plus .1ex minus .2ex}{0pt}
\usepackage{graphicx}
\usepackage{subcaption}
\usepackage{wrapfig}
\usepackage{amsfonts,amssymb,amsmath,amsthm,txfonts}
\usepackage{enumerate}
\usepackage{url}
\usepackage[mathcal,mathscr]{euscript}
\usepackage[dvipsnames]{xcolor}
\usepackage{hyperref}
\hypersetup{colorlinks=true,linkcolor=Magenta,citecolor=PineGreen}
\usepackage{float}
\usepackage{tikz-cd}
\usepackage[most]{tcolorbox}
\usepackage{epigraph}

\theoremstyle{plain}
\newtheorem{thm}[equation]{Theorem}

\newtheorem{lemma}[equation]{Lemma}

\theoremstyle{definition}

\newtheorem{rmk}[equation]{Remark}
\newcommand{\pp}{{\pmb{p}}}

\newcommand{\xx}{{\pmb{x}}}
\newcommand{\yy}{{\pmb{y}}}
\newcommand{\ff}{{\pmb{f}}}

\newcommand{\cc}{\pmb{c}}

\newcommand{\ta}{\mathrm{\mathrm{ta}}}

\newcommand{\PU}{\mathrm{\mathrm{PU}}}

\newcommand{\re}{\mathrm{\mathrm{Re}\,}}

\newcommand{\im}{\mathrm{\mathrm{Im}\,}}

\newcommand{\PP}{\mathbb{P}}
\newcommand{\HH}{\mathbb{H}}

\newcommand{\CC}{\mathbb{C}}

\newcommand{\ZZ}{\mathbb{Z}}
\newcommand{\KK}{\mathbb{K}}

\newcommand{\SP}{\mathbb{S}}
\newcommand{\RR}{\mathbb{R}}

\newcommand{\DD}{\mathbb{D}}

\title{\vspace{-15mm}\fontsize{16pt}{10pt}\selectfont\textbf{On the interplay between discrete invariants of complex hyperbolic disc bundles over surfaces}}

\author{
	\large
	\textsc{Hugo Cattarucci Botós}\thanks{Supported by  Max-Planck-Gesellschaft (MPG).}\\
	\normalsize \href{}{hugocbotos@gmail.com}\\
	\normalsize{Max Planck Institute for Mathematics, Bonn, Germany}	
	\vspace{-5mm}
}
\date{}

\begin{document}
	
\maketitle

    \begin{center}
	   \large\textbf{Abstract}
    \end{center}
    We investigate the relationship between three natural invariants of complex hyperbolic disc orbibundles over oriented and closed hyperbolic $2$-orbifolds. These invariants are the Euler characteristic $\chi$ of the $2$-orbifold, the Euler number $e$ of the disc orbibundle, and the Toledo invariant $\tau$ of a faithful representation of the surface group into $\PU(2,1)$ attached to the complex hyperbolic structure of the disc orbibundle. Based on previous examples, we conjecture that $-3|\tau| = 2e+2\chi$ always holds. For complex hyperbolic disc orbibundles over $2$-orbifolds derived from quadrangles of bisectors via tessellation, we prove that $3\tau = 2e+2\chi$. Furthermore, we demonstrate that $-3|\tau| = 2e+2\chi$ holds when a section with no complex tangent planes is present.
    
    \section{Introduction}
    Consider a complex hyperbolic disc bundle $L \to \Sigma$ over an oriented, connected, closed surface with a genus greater than one. When we say that $L$ has a complex hyperbolic structure, we mean that $L$ is a quotient of the complex hyperbolic plane $\HH_\CC^2$ by a discrete group of holomorphic isometries, which is isomorphic to $\PU(2,1)$.
    Since the fibers are contractible, this discrete group is isomorphic to the fundamental group $G$ of the surface $\Sigma$, that is, there exists a discrete faithful representation $\rho: G \to \PU(2,1)$ such that $L = \HH_\CC^2/G$.

    Once we have a complex hyperbolic disc bundle, there are three discrete invariants at play: the  Euler characteristic $\chi$ of the surface, the Euler number $e$ of the disc bundle, and the Toledo invariant $\tau$ of the representation $\rho$. If we see $\Sigma$ as a section of $L$, we can think of $e$ as the Euler number of its normal bundle and $\frac32\tau$ is the Euler number of the complex line bundle $\wedge^{\!2} TL|_\Sigma$ (see \cite{GKL}).

    \begin{rmk} We orient the disc bundle and the surface in a way that their orientation matches the orientation of the total space, which is naturally oriented because it is a complex manifold. To be more specific, when we embed $\Sigma$ as a section, each point $\xx$ on $\Sigma$ has an orientation determined by vectors $u_1$ and $u_2$ in the tangent space $T_\xx \Sigma$. Similarly, the orientation of the disc bundle $L$ at $\xx$ is determined by vectors $v_1$ and $v_2$ in the tangent space $T_\xx L_\xx$. The compatibility of these orientations means that the vectors $u_1$, $u_2$, $v_1$, and $v_2$ agree with the orientation of the complex manifold $L=\HH_\CC^2/G$. Here, if $e_1,e_2$ is a $\CC$-linear basis for $T_\xx L$, then $e_1,ie_1,e_2,ie_2$ provides its natural orientation.
 \end{rmk}

The mentioned discrete invariants are involved in two key conjectures. The first is the complex variant of the Gromov-Lawson-Thurston conjecture, stating that an oriented disc bundle has a complex hyperbolic structure if, and only if, $|e| \leq |\chi|$. The second conjecture states that a complex hyperbolic disc bundle admits a holomorphic section if, and only if, $3\tau = 2e+ 2\chi$.

 \begin{rmk}
     The GLT-conjecture was initially stated for hyperbolic structure instead of complex hyperbolic (see \cite{GLT}). The complex variant was then proposed in \cite{discbundles}.
 \end{rmk}

	The presence of a holomorphic section means that we have $\Sigma$ embedded as Riemann surface in the Kähler manifold $L$. In this situation, it is easy to see how $3\tau = 2e+ 2\chi$ appears: consider the isomorphism of complex vector bundles $T\Sigma \oplus N\Sigma \simeq TL|_\Sigma$, where $N\Sigma$ is the normal bundle of $\Sigma$. Since $T\Sigma$ and $N\Sigma$ are complex line bundles, we obtain the following identity
    $$c_1(TL|_\Sigma) = c_1(T\Sigma)+ c_1(N\Sigma),$$
    where $c_1$ stands for the first Chern number. Thus,
    $$c_1(TL|_\Sigma) = c_1(\wedge^{\!2}TL|_\Sigma) = \frac32 \tau$$ 
    and, as consequence, $3 \tau = 2e+2\chi$,
    because for complex line bundles the first Chern number equals the Euler number.

    \begin{rmk} Since $\tau$ can be computed from the symplectic form $\omega$ of $\HH_\CC^2$ as $$\tau = \frac{4}{2\pi} \int_\Sigma \omega,$$
    where $\omega$ is the imaginary part of the Hermitian metric of $\HH_\CC^2$ (see Section \ref{sec basics}), we conclude that $e\leq -\chi$, because $\tau$ is negative here. For details about the Toledo invariant, see \cite{tol}, \cite{bot}.
    \end{rmk}
    
    Similarly, the presence of an anti-holomorphic section also results in $-3\tau = 2e+2\chi$. Indeed, we can take the representation $\rho$ as previously described and conjugate it by an anti-holomorphic isometry, thus providing a new disc bundle with a holomorphic section. The Euler number and the Euler characteristic are unchanged under this procedure, but the Toledo invariant changes sign.

    Regardless of the situation, when holomorphic or anti-holomorphic sections exist, the identity $-3|\tau| = 2e+2\chi$ always holds.

    The curious thing is that, to the best of the author's knowledge, all examples of complex hyperbolic disc bundles satisfy $-3|\tau| = 2e+2\chi$ (see Section \ref{section: Known examples}).

\begin{tcolorbox} \begin{center} \textbf{Objectives of this paper}\end{center}
    We aim to shed some light on the identity 
     \begin{equation}\label{eq k}
     -3|\tau| = 2e+ 2\chi \tag{$\varheartsuit$}
     \end{equation}
    which frequently emerges in complex hyperbolic disc bundles. 
    
    More precisely:
    \begin{enumerate}
        \item We show that the identity $3 \tau=2e+2\chi$ holds for a widely used construction based on quadrangles of bisectors. This result is significant because it enables us to calculate the value of $e$ using the Toledo invariant, which is generally straightforward to compute, whereas evaluating $e$ directly can be challenging. The fact that $3\tau = 2e + 2\chi$ is true for constructions based on quadrangles of bisectors was first conjectured by Sasha Anan'in, Carlos H. Grossi, Nikolay Gusevskii on \cite{discbundles} based on thousands of examples and it is nicknamed Kalashnikov conjecture\footnote{The mention of Kalashnikov here is meant as a humorous reference. There is an anecdote that suggests every Soviet factory, regardless of its intended production, would inevitably produce Kalashnikov rifles unintentionally. Similarly, in our search for examples, we consistently encounter the result $3\tau = 2e+ 2\chi$.}.
        
        \item We prove that if a complex hyperbolic disc bundle admits a section with no complex tangent planes, then the identity \eqref{eq k} holds.
    \end{enumerate}
    
    \end{tcolorbox}

\subsection{Construction of examples}
 We avoid working directly with surfaces when building examples because their fundamental groups are too complicated. Instead, we employ oriented, connected, closed hyperbolic $2$-orbifolds to build complex hyperbolic disc orbibundles, then use that every closed hyperbolic $2$-orbifold is finitely covered by a closed hyperbolic surface\footnote{Selberg's lemma guarantees that every cocompact Fuchsian group admits a finite index torsion-free normal subgroup.}, and finally pullback the disc orbibundle to a disc bundle over a surface, which will inherit the complex hyperbolic structure. Additionally, the relative values $e/\chi$ and $\tau/\chi$ are unchanged under pullback, thus if $3 \tau = 2e+2\chi$, for instance, holds on the orbifold level, then the same happens for the derived surfaces (see \cite{bot} for a detailed exposition about orbibundles and the described invariants). 
    
Let $G$ be a cocompact Fuchsian group, i.e., the quotient space $\Sigma:=\HH_\CC^1/G$ is an oriented, closed, and connected hyperbolic $2$-orbifold.
Consider as well a faithful discrete representation $\rho: G \to \PU(2,1)$ such that the quotient $L=\HH_\CC^2/G$ is a disc orbibundle over $\Sigma$. 
    
    More precisely, we have a smooth action of $G$ on the $4$-ball $\HH_\CC^1 \times \DD^2$ of the form $$g(x,f) = (gx,a(g,x)f),$$ where $a(g,x)$ is a automorphism of the disc $\DD^2$ and depends smoothly of $x$, and a $G$-equivariant diffeomorphism $\HH_\CC^1 \times \DD^2 \to \HH_\CC^2$. With these ingredients, we have the disc orbibundle $$\HH_\CC^2/G \simeq (\HH_\CC^1 \times \DD^2)/G \to \HH_\CC^1/G,$$
where this last map is simply $[x,f] \mapsto [x]$. See \cite[Euler number for orbigoodles]{bgr} for more details.
    
To construct such complex hyperbolic disc orbibundles, one usually builds a fundamental domain $\bf Q$ in the complex hyperbolic plane for the action of the fundamental group $G$ of $\Sigma$ and then fibers $\bf Q$ by discs in such a way that this fibration can be extended to the whole complex hyperbolic space via tessellation.

       \subsection{Known examples}
       \label{section: Known examples}
We provide a list of all complex hyperbolic disc bundles to the author's knowledge. All of them satisfy the equation \eqref{eq k}.

The first examples ever constructed were those in \cite{GKL}. The authors produced one complex hyperbolic disc bundle for each possible even Toledo number $\tau$ and we believe they satisfy $-3|\tau| = 2e + 2\chi$ by construction. This identity is not the one found in their paper. It is a small correction we propose in their computation of the Euler number (see Remark \ref{GKL euler} for details). 

For the remaining examples mentioned in this list, $3 \tau = 2e+2\chi$ and $\tau\leq 0$ were observed empirically. 
    
    The turnover group is defined as $$G(n_1,n_2,n_3)=\langle g_1,g_2,g_3: g_1^{n_1} = g_2^{n_2}=g_3^{n_3}=g_3g_2g_1 = 1 \rangle,$$ where $n_1,n_2,n_3$ are positive integers satisfying $\frac1{n_1}+\frac1{n_2}+\frac1{n_3}<1$. 
It is the fundamental group of the hyperbolic $2$-orbifold $\SP^2(n_1,n_2,n_3)$, a sphere with $3$ conic points of angles $\frac{2\pi}{n_1},\frac{2\pi}{n_2},\frac{2\pi}{n_3}$. In essence, the turnover is the simplest cocompact Fuchsian group.

    The turnover group was used in \cite{discbundles} with $n_1=n_2$ and $n_3=2$ to obtain thousands of examples of disc orbibundles, all of then satisfying $0<e/\chi <1/2$. 
    
    Each complex hyperbolic disc orbibundle up to isomorphism corresponds to a point of the $\PU(2,1)$-character variety of $G(n_1,n_2,n_3)$, the space of all faithful representations $G(n_1,n_2,n_3) \to \PU(2,1)$ modulo $\PU(2,1)$-conjugation. The examples in \cite{discbundles} are rigid, i.e., they form isolated points in the corresponding character variety, which means it is not possible to deform their complex hyperbolic structure. Misha Kapovich used this rigidity to prove that these examples actually admit holomorphic sections (see~\cite[Example 8.10]{kap2}).

    Non-rigid examples with similar construction are found in \cite{gaye}, where $0<e/\chi<1/2$ as well.     
    
    The turnover group with parameters $n_1,n_2,n_3$ with no restriction is used to construct new examples in \cite{bgr}. Unlike \cite{discbundles}, here we find hundreds of non-rigid examples. 
    Additionally, the space of the complex hyperbolic structures of each disc orbibundle is two-dimensional. The examples here have relative Euler number $e/\chi$ spread on the interval $[-1,1/2)$, including cotangent orbibundles ($e=-\chi$) and trivial orbibundles ($e=0$). We also found rigid examples with negative relative Euler numbers and $e=0$. It is unknown if these non-rigid examples have holomorphic sections.

    The example in \cite{sashanicolai} was the first example of a trivial bundle ($e=0$). It was constructed using representation of the hyperelliptic group $$H_5=\langle r_1,r_2,r_3,r_4,r_5: r_5r_4r_3r_2r_1 = r_i^2 =1\rangle,$$ the fundamental domain of $\SP^2(2,2,2,2,2)$, into $\PU(2,1)$. Currently, Felipe de Aguilar Franco and I are building non-rigid trivial bundles using the same techniques.

\begin{rmk}\label{GKL euler}
 The formula present in \cite{GKL} is $e-\chi = |\tau|/2$. As mentioned before, we suspect that their formula for the Euler number has a minor error. They used that the Euler number of the normal bundle of a Lagrangian submanifold is $\chi$, which is false. The correct value is $ -\chi$. Indeed, if $\Sigma$ is an oriented, connected, closed Lagrangian submanifold of a Kähler $4$-manifold, then the Euler number of its normal bundle is $-\chi(\Sigma)$ because there is a natural orientation-reversing isomorphism between the tangent and normal bundles of $\Sigma$, given by $v \mapsto iv$. Another proof of this fact is given at Lemma \ref{lagrangian section}. 
 
Now we explain why the identity $-3|\tau| = 2e+2\chi$ holds in their examples. In \cite[4.2. Calculation of Euler number and Toledo invariant]{GKL}, we have a disc bundle $M$ over the oriented, connected, closed surface $\Sigma$ with genus $g>1$. The surface $\Sigma$ is embedded in the complex hyperbolic disc bundle as a piecewise totally geodesic surface. There are two disjoint circles in $\Sigma$ separating it in two surfaces $\Sigma_1$ and $\Sigma_2$. Thus, each $\Sigma_k$ is a surface having two disjoint circles as the boundary. 

The surface $\Sigma_1$ is embedded as a complex geodesic component and $\Sigma_2$ is embedded as a Lagrangian component.

The correct formula for the Euler number is $$e=\chi(\Sigma_1)/2-\chi(\Sigma_2),$$ because $\Sigma_1$ is embedded as complex geodesics and, as consequence, the contribution from its normal bundle is $\chi(\Sigma)/2$. Similarly, $-\chi(\Sigma_2)$ is the contribution arising from the normal bundle over the Lagrangian part.

From their construction, the genus of $\Sigma_1$ is $g_1:=g-1-t$ and the genus of $\Sigma_2$ is $g_2:=t$, where $t:=g-1-|\tau|/2$, a non-negative integer because $|\tau|\leq |\chi|$ (Toledo's rigidity). Now, the Euler characteristic of $\Sigma_k$ is given by $$2-2g_k = \chi(\Sigma_k) + \chi(\DD^2 \sqcup \DD^2) - \chi(\SP^1 \sqcup \SP^1)= \chi(\Sigma_k) + 2,$$ 
because by gluing discs to the boundaries of $\Sigma_k$ we obtain a closed surface of genus $g_k$. Simplifying the above formula we obtain $\chi(\Sigma_k) = -2g_k$ and, as consequence,
\begin{align*}
e&=\chi(\Sigma_1)/2-\chi(\Sigma_2) 
\\&= -g_1+2g_2
\\&= -g+1 + 3t
\\& =2g-2 -\frac{3|\tau|}2.
\end{align*}

Therefore, $2e+2\chi = -3|\tau|$.
\end{rmk}
 
    \section{The complex hyperbolic plane, bisectors, and quadrangles}
    \label{sec basics}
    We present basic facts about real and complex hyperbolic geometry from the projective viewpoint following the works \cite{discbundles}, \cite{coordinatefree}, and \cite{goldmanbook}.
    
    Let $\KK = \RR$ or $\CC$. A $(n+1)$-dimensional $\KK$-linear space $V$ endowed with a Hermitian form $\langle - , -  \rangle $ with signature $-+\cdots+$ gives rise to the $n$-dimensional $\KK$-hyperbolic space
    $$\HH_\KK^n = \{\xx \in \PP_\KK(V): \langle x, x\rangle <0\},$$
    where we use $x \in V\setminus\{0\}$ to denote a representative for $\xx \in \PP_\KK(V)$. As a metric space, the distance between two points in $\HH_\KK^n$ is given by $\cosh^2\left(d(\xx,\yy)\right)=\ta(\xx,\yy)$, where the function $$\ta(\xx,\yy) := \frac{\langle x,y \rangle \langle y,x \rangle }{\langle x,x \rangle \langle y,y \rangle}.$$  
    is called {\it tance}. 
    
    From a differential geometry viewpoint, we have the natural isomorphism $T_\xx \HH_\KK \simeq \hom(\KK x, x^\perp)$ and the Hermitian metric
    $$\langle u,v \rangle = -\frac{\langle u(x),  v(x) \rangle}{ \langle x,x \rangle}, \quad  u,v \in T_\xx \HH_\KK^n.$$

    In the real case, the Hermitian metric is a Riemannian metric of constant sectional curvature $-1$. In the complex case, $\HH_\CC^n$ is a Kähler manifold with Riemannian metric $g := \re \langle -,- \rangle$ and symplectic form $\omega:= \im \langle -, - \rangle$. The sectional curvature of $\HH_\CC^n$ assume all values in the interval $[-4,-1]$ when $n>1$. For $n=1$, the space $\HH_\CC^1$ is a Poincaré disc with curvature $-4$.

    A geodesic $G$ is described as the projectivization of a real $2$-dimensional subspace $W$ of $V$ such that $\langle-,-\rangle|_{W\times W}$ is real-valued and has signature $-+$. More precisely, $G= \PP_\KK(W) \cap \HH_\KK^n$. In the complex hyperbolic geometry, we also have the complex geodesics, non-trivial intersections of complex projective lines with $\HH_\CC^n$. It is important to note that every geodesic $G$ is contained in one, and only one, complex geodesic $L$. Indeed, if $G= \PP_\CC(W) \cap \HH_\CC^2$, then $L= \PP_\CC(W\oplus iW) \cap \HH_\CC^2$.

    The complex hyperbolic plane is $\HH_\CC^2$ and it admits $3$ types of non-trivial totally geodesic submanifolds: geodesics, complex geodesics, and real planes. A real plane is an isometrically embedded Beltrami-Klein model $\HH_\RR^2$ of the form $\PP_\CC(W) \cap \HH_\CC^2$, where $W \subset V$ is a $3$-dimensional real subspace of $V$ such that $\langle -, - \rangle$ restricted to $W$ is real-valued and has signature $-++$. Observe that each complex geodesic is a $\HH_\CC^1$ embedded as a Riemann surface and each real plane is a $\HH_\RR^2$ embedded as a Lagrangian submanifold. Additionally, there are no $3$-dimensional totally geodesic submanifolds, a misfortune, because they would be natural candidates to build fundamental domains. Nevertheless, we have bisectors. 

    Complex geodesics are obtained from the projectivization of two-dimensional complex subspaces of $V$ with signature $-+$. Thus, complex geodesics are always of the form $\PP(p^\perp) \cap \HH_\CC^2$, where $\pp \in \PP_\CC(V)$ and $\langle p,p \rangle >0$, that is, $\pp$ is a positive point. Two distinct complex geodesics $L_1, L_2$, with $L_i = \PP(p_i^\perp) \cap \HH_\CC^2$, are said to be ultraparallel, asymptotic, concurrent if the point $\PP(p_1^\perp) \cap \PP(p_2^\perp)$ is positive, null, negative, respectively. Algebraically, the two complex geodesics are ultraparallel, asymptotic, concurrent if $\ta(\pp_1,\pp_2)>1$, $\ta(\pp_1,\pp_2)=1$, $\ta(\pp_1,\pp_2)<1$, respectively. 
    
    Bisectors appear naturally when constructing Dirichlet domains, where  hypersurfaces like
    $$\{\xx \in \HH_\CC^2: d(\pp,\xx)=d(\xx,\pp')\}$$
    are used to construct the boundary for the domain. Here $\pp,\pp'$ are distinct points of $\HH_\CC^2$.

    A hypersurface like the one above is called a bisector. Nevertheless, this definition is quite complicated to manipulate because different pairs of points $\pp,\pp'$ can produce the same bisector.

    An algebraic way of defining such an object is the following: Consider a geodesic $G=\PP_\CC(W)\cap \HH_\CC^2$, where $W$ is real $2$-dimensional subspace of $V$ such that $\langle-,-\rangle|_{W\times W}$ is real-valued and has signature $-+$. The complex geodesic $L = \PP_\CC(W\oplus i W) \cap \HH_\CC^2$ contains $G$. Since $W+iW$ is a two-dimensional complex subspace of $V$, it possesses a unit vector $f$ orthogonal to it by projective duality, i.e., $f^\perp = W+iW$. Additionally,  $\langle f,f \rangle = 1$ because $W\oplus iW$ has signature $-+$ and the signature os $V$ is $-++$.
    The bisector defined by $G$ is given by $B=\PP_\CC(W+\CC f)\cap\HH_\CC^2$ and, topologically, it is a cylinder:
    $$B = \bigsqcup_{\xx \in G} \PP_\CC(\CC x + \CC f)\cap \HH_\CC^2.$$
    The geodesic $G$ is called the real spine of the bisector, the complex geodesic $L$ is its complex spine, and $\ff$ is its polar. The complex geodesics $\PP_\CC(\CC x + \CC f)\cap \HH_\CC^2$, with $\xx\in G$, are called slices.

    The boundary of a fundamental domain constructed using bisectors is formed by segments of bisectors: if we have two ultraparallel complex geodesics $C_1,C_2$, then there exists a unique geodesic $G$ orthogonal to both complex geodesics. The geodesic intersect $C_1$ and $C_2$ in points $\cc_1$ and $\cc_2$ and the segment of geodesic $G[\cc_1,\cc_2]$ connect the two complex geodesics. Since $C_1,C_2$ are ultraparallel, their corresponding projective lines intersect in a positive point $\ff$. Thus we have the bisector defined by $G$ with $\ff$ as polar. The segment of bisector connecting $C_1$ to $C_2$ is just 
        $$B[C_1,C_2] = \bigsqcup_{\xx \in G[\cc_1,\cc_2]} \PP(\CC x + \CC f).$$

    With these ingredients, we are prepared to talk about quadrangles of bisectors.
    Consider four pairwise ultraparallel complex geodesic $C_1,C_2,C_3,C_4$. We can define the quadrangle
    \begin{equation}\tag{$\blacklozenge$}\label{quadrangle}
    \mathcal Q = B[C_1,C_2] \cup B[C_2,C_3] \cup B[C_3,C_4] \cup B[C_4,C_1]
    \end{equation}
    and on the right conditions $\mathcal Q$ bounds a $4$-ball.

    	\begin{figure}[H]
		\centering
		\begin{minipage}{.5\textwidth}
			\centering
	           \includegraphics[scale = .7]{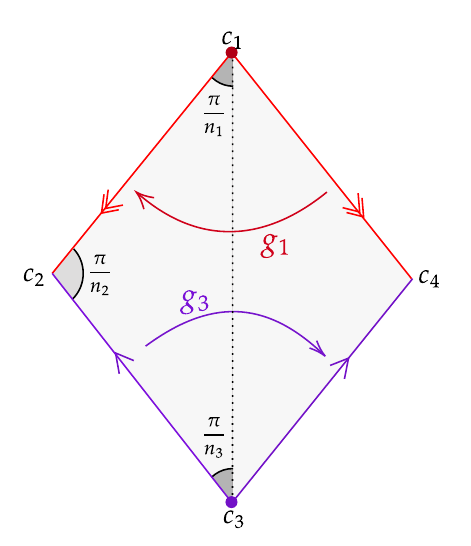}
		\end{minipage}%
		\begin{minipage}{.5\textwidth}
			\centering
			\includegraphics[scale =0.6]{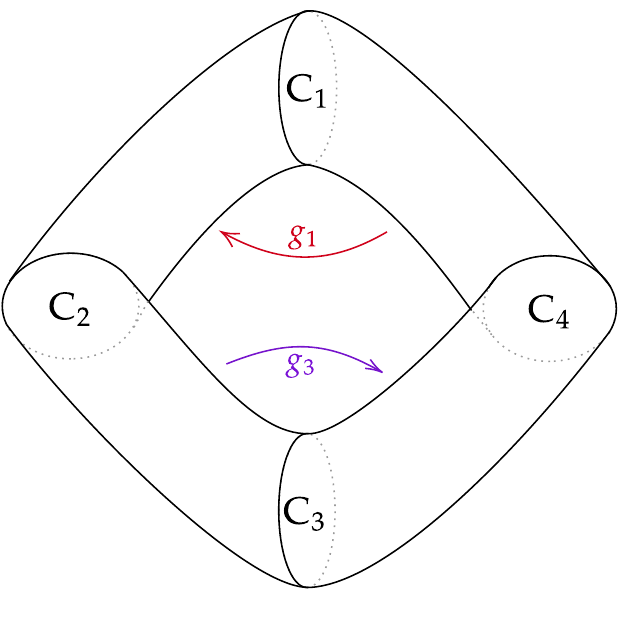}
		\end{minipage}
		\caption{ Turnover fundamental domain in $\HH_\CC^1$ and $\HH_\CC^2$.}
		\label{turnover fundamental domain}
	\end{figure}

    The quadrangle is commonly used to build complex hyperbolic disc orbibundle since its boundary is already fibered by discs. The examples in \cite{discbundles}, \cite{bgr}, \cite{sashanicolai}, and \cite{gaye} were constructed using quadrangles. The reason for this common use is not really a surprise. If we have to build examples, it is better to construct over $2$-orbifolds, since they have simple fundamental domains and the idea is to mimic those inside $\HH_\CC^2$. 
 
    The fundamental domain for the turnover group is just a quadrilateral like the one described in the Figure~\ref{turnover fundamental domain}: we construct a triangle with vertices $\cc_1,\cc_2,\cc_3$ with inner triangle $\pi/n_1,\pi/n_2,\pi/n_3$ and reflect it with respect to the line $\cc_1,\cc_3$, thus obtaining a quadrilateral. The isometries $g_1$ and $g_3$ are just rotations centered at $\cc_1,\cc_3$ by angles $-2\pi/n_1$ and $-2\pi/n_3$ respectively. The isometry $g_2$ is then defined as $g_3^{-1}g_1^{-1}$ and it is indeed a rotation of angle $-2\pi/n_2$ centered at $\cc_2$. 

    Now assume we have a faithful representation of $G(n_1,n_2,n_3) \to \PU(2,1)$.
    The corresponding version of the fundamental domain for the action of the turnover on $\HH_\CC^2$ (Figure \ref{turnover fundamental domain} to the right) is formed by four segments of bisectors connecting the vertices $C_1,C_2,C_3,C_4$. Here $C_i$ is a complex geodesic stable under action of $g_i$ for $i=1,2,3$ and $C_4 := g_1^{-1}C_2 = g_3 C_2$. Under the right conditions, which we discuss in Section~\ref{sec: geo quadrangle}, this quadrangle bound a ball fibered by discs respecting the fibration of the boundary and tessellate $\HH_\CC^2$, thus forming a complex hyperbolic disc orbibundle $\HH_\CC^2/G \to \HH_\CC^1/G$.

    \begin{figure}[H]
		\centering
		\begin{minipage}{.5\textwidth}
			\centering
	           \includegraphics[scale = .7]{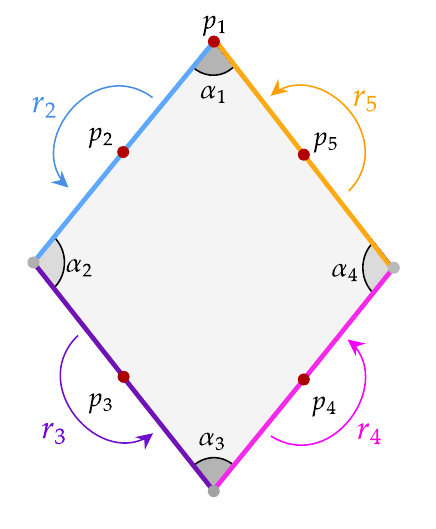}
		\end{minipage}%
		\begin{minipage}{.5\textwidth}
			\centering
			\includegraphics[scale =0.6]{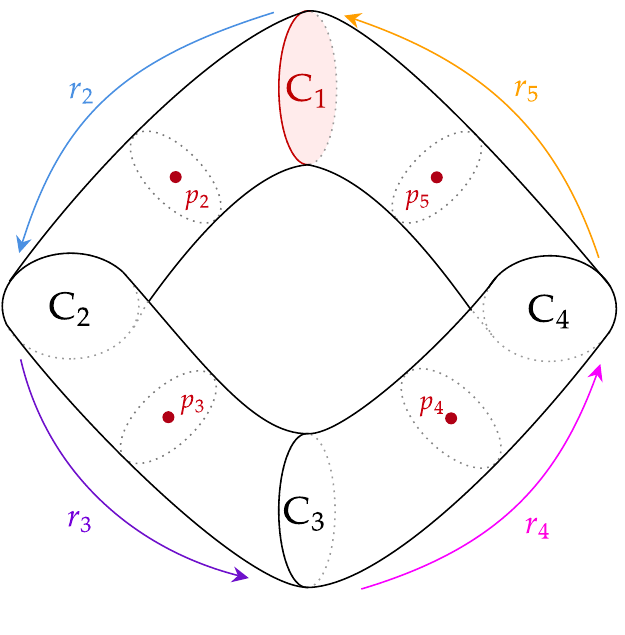}
		\end{minipage}
		\caption{ hyperelliptic $H_5$ fundamental domain in $\HH_\CC^1$ and $\HH_\CC^2$.}
		\label{hyper fundamental domain}
	\end{figure}
    
    The fundamental group for the hyperelliptic group $H_5$ is the quadrilateral to the left in the Figure \ref{hyper fundamental domain}, where we assume that the sum of the inner angles $\alpha_i$ total $\pi$. The isometries $r_2,r_3,r_4,r_5$ are reflections at $\pp_2,\pp_3,\pp_4,\pp_5$, where $\pp_i$, with $i \neq 1$, is the middle point of a geodesic segment as described in the Figure~\ref{hyper fundamental domain} and $\pp_1$ is the vertex at the top. Due to the choices of angles, $r_1:=r_5 r_4 r_3 r_2$ is a reflection at $\pp_1$.
    The candidate for a fundamental domain in $\HH_\CC^2$ for a faithful representation $H_5 \to \PU(2,1)$ is the following: take $\pp_1$ to be a positive point and $\pp_2,\pp_3,\pp_4,\pp_5$ to be negative points. Define the reflections
    $$r_i(x) = -x+ 2\frac{\langle x,p_i\rangle}{\langle p_i,p_i\rangle} p_i.$$
    The isometries $r_2,r_3,r_4,r_5$ are reflections at points $\pp_i$ in $\HH_\CC^2$. On the other hand, $r_1$ is a reflection at the complex geodesic $C_1:=\PP_\CC(p_1^\perp) \cap \HH_\CC^2$. If we are able to find points $\pp_1,\ldots,\pp_5$ such that $r_5r_4r_3r_2r_1 = 1$ in $\HH_\CC^2$, then we can construct $C_2 = r_2 C_1$, $C_3=r_3 C_2$, $C_4 = r_4 C_3$ and under the right conditions we have a tessellation of $\HH_\CC^2$. This technique was the one used on the paper \cite{sashanicolai} to build the first instance of complex hyperbolic trivial disc bundle, and Felipe Franco and I have been exploring faithful representations of $H_5 \to \PU(2,1)$ using the same construction.

    \newpage

    \section{The geometry of the quadrangle}\label{sec: geo quadrangle}

\begin{wrapfigure}[17]{r}{0.53\textwidth}
 
        \includegraphics[width=0.47\textwidth]{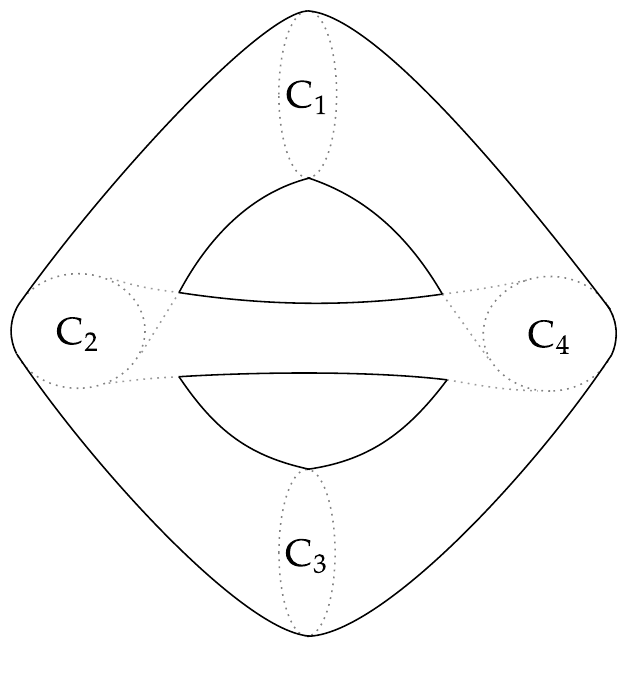}
        \vspace{0.19cm}
        \caption{Two adjacent triangles of bisectors.}
    \end{wrapfigure}
    
    Consider the quadrangle $\mathcal Q$ as defined in \eqref{quadrangle}. We now lay down the conditions for this quadrangle to bound a ball and be fibered by discs.
    The terms in italics will be explained next.

    \begin{itemize}
        \item[{\bf K1}] All pairs of vertices $C_i$ must be  ultraparallel to one another. 
        \item[{\bf K2}] The triangles of bisectors $\triangle (C_1, C_2, C_4)$ and $\triangle (C_3, C_4,C_2)$ must be {\it transversal} and {\it counterclockwise-oriented}.
        \item[{\bf K3}] The triangles of bisectors  $\triangle (C_1, C_2, C_4)$ and $\triangle (C_3, C_4,C_2)$ must be {\it transversally adjacent}.
    \end{itemize}

    Now, we detail the terms used above. 

    \medskip
    
    Let $\pp_1,\pp_2,\pp_3,\pp_4$ be the polar points of the projective lines defining $C_1,C_2,C_3,C_4$.

    \medskip

    \begin{rmk} The intersection of two transversal bisectors is always a complex geodesic, slice shared by both bisectors (see \cite{discbundles}).
    \end{rmk}
    
    The {\bf triangle of bisectors $\triangle (C_1, C_2, C_3)$ is transversal} when adjacent bisectors (seemed as extended hypersurfaces) intersect transversally.  Define $$t_{ij} := \sqrt{\ta(p_i,p_j)}\quad \text{and}\quad \varepsilon_0+ i \varepsilon_1 := \frac{ \langle p_1,p_2 \rangle\langle p_2,p_3 \rangle\langle p_3,p_1 \rangle}{|\langle p_1,p_2 \rangle\langle p_2,p_3 \rangle\langle p_3,p_1 \rangle|}.$$
    Algebraically, to say the triangle $\triangle (C_1, C_2, C_3)$ is transversal mean that \begin{align*}
    \varepsilon_0^2 t_{12}^2 + t_{23}^2+t_{31}^2&<1+2t_{12}t_{23}t_{31} \varepsilon_0,\\
    \varepsilon_0^2 t_{31}^2 + t_{12}^2+t_{23}^2&<1+2t_{12}t_{23}t_{31} \varepsilon_0,\\
    \varepsilon_0^2 t_{23}^2 + t_{31}^2+t_{12}^2&<1+2t_{12}t_{23}t_{31} \varepsilon_0.
    \end{align*}
    
    To say that such a transversal triangle is {\bf counterclockwise-oriented} means that $\varepsilon_1<0$. To understand the meaning of this statement it is important to understand that the space of all transversal and counterclockwise-oriented triangles of bisectors is path-connected, which means we can deform one on the other (see \cite[Lemma A. $31$]{discbundles}), and among then there is a very special one, the triangle of bisectors over a complex geodesic: consider a complex geodesic $L = \PP_\CC(f^\perp)\cap \HH_\CC^2$, where $f$ is a positive point, and a geodesic triangle with vertices $c_1,c_2,c_3$ in $L$ such that the path going along $c_1,c_2,c_3$ forms a counterclockwise loop. In this case, the triangle of bisectors defined by $C_i = \PP_\CC(\CC c_i \oplus \CC f)$ is transversal and counterclockwise-oriented with $\varepsilon_1 <0$. In fact, its value is  $\varepsilon_0+i \varepsilon_1 = \exp\left(-2\,\textbf{area}(\triangle(c_1,c_2,c_3) i\right)$ with $\textbf{area}(\triangle(c_1,c_2,c_3)) < \pi/4$ because the curvature of a complex geodesic is $-4$. Thus, to be a transversal and counterclockwise-oriented triangle of bisectors just mean it was obtained from this simple configuration over a complex geodesic.
    Observe that the conditions {\bf K1} and {\bf K2} guarantee that each triangle of bisector bounds an open $4$-ball. Furthermore, its boundary inside $\HH_\CC^2$ is a solid torus, since each bisector is a cylinder.

    To say that the triangles $\triangle (C_1, C_2, C_4)$ and $\triangle (C_3, C_4, C_2)$ are {\bf transversally adjacent} mean that the bisectors for $B(C_1,C_2)$ and $B(C_3,C_2)$ intersect transversally at $C_2$, the bisectors for $B(C_1,C_4)$ and $B(C_3,C_4)$ intersect transversally at $C_4$, and $C_3$ is inside the sector defined by the segments of bisectors $B[C_1,C_2)$ and $B[C_1,C_4)$ starting at $C_1$ and extended pass $C_2$ and $C_4$ (see subsection \cite[6.1. Quadrangle of bisectors]{bgr} and \cite{discbundles}). Additionally, under {\bf K1}, {\bf K2}, and {\bf K3}, the quadrangle $\mathcal Q$ bounds an open $4$-ball in the complex hyperbolic plane. The union of the quadrangle $\mathcal Q$ with the ball bounded by it forms a domain $\bf Q$ in $\HH_\CC^2$.

    Following Lemma A. $31$ from \cite{discbundles}, there is an isotopy between two transversal and counterclockwise-oriented triangles of bisectors in such a manner that in each step of the isotopy the region is a transversal and counterclockwise-oriented triangle. Adapting this result, if we have a quadrangle $\mathcal Q$ satisfying the three conditions {\bf K1}, {\bf K2}, and {\bf K3} and $\ff$ is the polar of the bisector $B[C_2,C_4]$ dividing the quadrangle in half, then we can deform the quadrangle to a quadrangle $\mathcal Q'$ over the complex geodesic $\PP(f^\perp)\cap \HH_\CC^2$ by deforming each triangle $\triangle (C_1, C_2, C_4)$ and $\triangle (C_3, C_4,C_2)$ separately always keeping $B[C_2,C_4]$ unmoved. Thus we have an isotopy between $\bf Q$ and $\bf Q'$, a quadrangle over a complex geodesic. This isotopy is formed by quadrangles satisfying {\bf K1} and {\bf K2} at every step (see Subsection \cite[7.2. Deformation lemma]{bgr}).

    More precisely, there is an isotopy $F_t: \textbf Q \to \HH_\CC^2$ such that $F_0$ is the inclusion and $F_1$ maps $Q$ to $Q'$. For each $t$, the image $F_t(\mathcal Q)$ is a quadrangle of bisectors, each bisector is mapped to a bisector, each vertex of $\mathcal Q$ is mapped to a vertex of $F_t(\mathcal Q)$, each slice is mapped to a slice isometrically. 

    Since the ball $\bf Q'$ is naturally fibered by discs, we can transport such fibration to $\bf Q$. Thus, we reached the required fibration necessary for the construction of complex hyperbolic disc orbibundles.
    \newpage
    \section{The Kalashnikov}
    
    The Kalashnikov conjecture states that every fundamental domain created using a quadrangle of bisectors as described in section \ref{sec: geo quadrangle} leads to the identity $3\tau = 2 e + 2 \chi$, where, by convention, we orient the disc orbibundle in agreement with the orientation of the discs fibering the quadrangle (if we orient the disc orbibundle with opposite orientation, then $-3\tau = 2e+ 2\chi$).

    Our objective here is to prove the validity of such conjecture. Consider a quadrangle $\mathcal Q$ satisfying {\bf K1}, {\bf K2}, and {\bf K3}, a cocompact Fuchsian group $G$ and a discrete faithful representation $G\to \PU(2,1)$ such that the quadrangle defines a fundamental domain for the complex hyperbolic disc orbibundle $\HH_\CC^2/G \to \HH_\CC^1/G$. Let us denote the total space $\HH_\CC^2/G$ by $L$ and the base space $\HH_\CC^1/G$ by $\Sigma$.

    We can see $\Sigma$ as a section of $L \to \Sigma$. On the complex hyperbolic plane, we consider a lift $\tilde \Sigma$ for this section, meaning that $\tilde \Sigma$ is $G$-invariant and $\Sigma = \tilde \Sigma/G$. We can also identify $L$ with a rank $2$ oriented vector orbibundle by identifying $L_\xx$ with $T_\xx L_\xx$, with  $\xx \in \Sigma$, that is, we identify each $L_\xx$ with the plane tangent of $L_\xx$ at~$\xx \in \Sigma$.
    
    Note that $\tilde \Sigma$ can be seen as a $G$-equivariant smooth embedding of the disc $\HH_\CC^1$ into $\HH_\CC^2$. Thus, the polygon $P:= \mathbf Q \cap \tilde \Sigma$ is a fundamental domain for $\HH_\CC^1/G$.

    \begin{rmk}For an $n$-dimensional $\RR$-linear space $V$, the quotient $V/\RR_{>0}$ is an $(n-1)$-sphere.
    
    Similarly, if $E \to \Sigma$ is a rank $2$ oriented vector orbibundle, we denote by $\SP(E) \to \Sigma$ its associated circle bundle, where $\SP(E)_\xx = E_\xx/ \RR_{>0}$. 
    \end{rmk}

    Consider two unit vector fields  $u,v$ over $P$, where $u$ is tangent to $P$ and $v$ is normal to $P$. 
    
    \begin{lemma}[\bf Extension Lemma]\label{extension lemma} The section $u \wedge v$ of $\wedge^{\!\!2} T\HH_\CC^2$ defined over $\partial P$ can be extended to a non-vanishing section on $P$. 
    \end{lemma}
    \begin{proof} 
    
    Let $F_t$ be the isotopy described by the end of Section \ref{sec: geo quadrangle}. We need to show that the loop 
    \begin{align*}
    \lambda_0:\partial P &\to \SP\left(\wedge^{\!\!2} T\HH_\CC^2\right)\\
    \xx &\mapsto  u(\xx)\wedge v(\xx)
    \end{align*} is contractible, i.e., we need to show that $$\lambda_0 \simeq 0\quad \text{in} \quad \mathrm H_1\left(\SP\left(\wedge^{\!\!2} T\HH_\CC^2\right),\ZZ\right).$$

    Note that $$u_t(x) = \mathrm d F_t \left(u\left(F_t^{-1}(x)\right)\right) \quad \text{and}\quad v_t(x) = \mathrm d F_t \left(v\left(F_t^{-1}(x)\right)\right)$$ define non-zero vector for each $t \in [0,1]$ and $x \in \partial F_t(P)$. Additionally, $u_t(\xx) \wedge v_t(\xx)$ never vanishes for $t \in [0,1]$ and $\xx \in \partial F_t(P)$, because $v_t(\xx)$ is always tangent to a slice (complex geodesic) of the bisectors forming the quadrangle $F_t(\mathcal Q)$ and $u_t(\xx)$ is not. 
    
    Thus, the section
    \begin{align*}
    \lambda_1:F_1(\partial P) &\to \SP(\wedge^{\!\!2} T\HH_\CC^2)\\
    \xx &\mapsto  u_1(\xx)\wedge v_1(\xx)
    \end{align*} 
    defines a loop in $\SP(\wedge^{\!\!2} T\HH_\CC^2)$ equal to $\lambda_0$ in homology.

    Since $u_1$ and $v_1$ extend over $F_1(P)$ via the isotopy, we have $u_1 \wedge v_1$ defined on $F_1(P)$. Note that this section does not vanish, because $u_1$ and $v_1$ are non-vanishing and $u_1$ is tangent to a complex geodesic to which $v_1$ is never tangent. Thus, we conclude $u_1 \wedge v_1$ does not vanish.
    Therefore, $\lambda_0  \simeq 0$ in homology. Thus, $u \wedge v|_{\partial P}$ extends to a non-vanishing section over $P$.
    
    \end{proof}

    Now we use the following trick. We have the rank $2$ oriented vector orbibundle $L$ over $\Sigma$. Let $\bf T$ be the vector orbibundle normal to $L$ in $T(\HH_\CC^2/G)|_{\Sigma}$. Note that $\bf T$ is a rank $2$ oriented vector orbibundle isomorphic to the tangent space $T\Sigma$. Thus, we identify $T\Sigma$ with $\bf T$ and, as consequence, the Euler characteristic of $\Sigma$ is the Euler number of $\bf T$. 
    
    \begin{thm} The identity $3\tau = 2e+2\chi$ holds if we orient the orbibundle in accordance with the discs fibering the quadrangle. That is, the Kalashnikov conjecture is true.
    \end{thm}
    \begin{proof} Consider unit vector fields $u,v$ over $P$ such that $u$ is section of $\bf T$ and $v$ is a section of $L$. Let $\eta := u \wedge v$ on $\partial P$. As consequence of the Extension Lemma \ref{extension lemma}, we can extend $\eta$ to $P$ as a section such that $\langle \eta, \eta \rangle = 1$.

    Let $v_1 =v$ and $v_2$ be the rotation of $v$ in $L$ by $90$ degrees in the counter-clockwise orientation. Do the same with respect to $T$: we take $u_1=u$ and $u_2$ is the $90$ degrees rotation of $u$ in $T$. For $\xx \in \partial P$, $L_\xx$ and ${\bf T}_\xx$ are complex lines. Thus, we either have $u_2=iu$ and $v_2 = iv$ or $u_2 = -iu$ and $v_2 = -i v$ on $\partial P$. 
    For $\eta$ as section of $\wedge^{\!\!2} T\HH_\CC^2$, the rotation of $\eta$ by $90$ degrees in the counter-clockwise direction is simply $i \eta$. 
    
    \begin{align*}
    \langle\nabla i\eta, \eta \rangle &= \langle (\nabla i u) \wedge v + u \wedge (\nabla i v), u \wedge v \rangle\\
    &=\langle (\nabla i u) \wedge v,u \wedge v\rangle  + \langle u \wedge (\nabla i v),u\wedge v \rangle\\
    &= \langle \nabla iu, u \rangle + \langle \nabla iv, v \rangle.
    \end{align*}
    Thus, taking the real part we obtain
    $$g(\nabla i\eta, \eta) =g(\nabla iu, u) + g(\nabla iv, v),$$
    and by Stokes theorem
    $$ \int_P \mathrm d g(\nabla u_2, u_1) + \int_P \mathrm d g(\nabla v_2, v_1)=\pm \int_P \mathrm d g(\nabla i\eta, \eta).$$

    Therefore, $2 e+2\chi = \pm 3\tau.$ 
    
    Note that by orienting the disc orbibundle and the discs fibering the quadrangle in a compatible way, we have $u_2=iu$ and $v_2=iv$ and, as consequence, $3\tau = 2e+2\chi$. Otherwise, we obtain $-3\tau = 2e + 2\chi$.
    \end{proof}

    \newpage
    \section{Avoiding complex points}

    Now we prove that if complex hyperbolic disc bundles admit a section with no complex tangent planes, then the equation \eqref{eq k} holds. We conjecture that all complex hyperbolic disc bundles satisfy the equation \eqref{eq k} and we believe that an adaptation of the argument below might be a way to prove such a statement. 
    
    Consider an arbitrary complex hyperbolic disc bundle $L \to \Sigma$, where now we take $\Sigma$ to be an oriented and connected closed surface with genus greater than one instead of an orbifold. We can assume that $\Sigma$ is embedded as a section of $L$. Let $N$ be the normal bundle of $\Sigma\subset L$. As bundles, $L\to \Sigma$ and $N \to \Sigma$ are isomorphic.

    The following lemma is well-known in symplectic geometry.
    \begin{lemma}\label{lagrangian section} If $\Sigma$ is an oriented, connected, closed Lagrangian surface of a $4$-dimensional complex hyperbolic manifold $L$, then $e(N)=-\chi$, where $e(N)$ is the Euler number of the normal bundle $N$ of $\Sigma$. 
    
    Additionally, whenever $L$ is a disc bundle and $\Sigma$ is a section, we have that $\tau =0$ and, as consequence, $-3|\tau| = 2e(L)+2\chi$ because the normal bundle is isomorphic to the disc bundle under such circumstance.
    \end{lemma}
    \begin{proof} Let $u_1,u_2$ be a positively oriented orthonormal frame of $T\Sigma$ with respect to the Riemannian metric $g$. Since $\Sigma$ is Lagrangian, these two vector fields are orthogonal concerning the Hermitian metric as well. Thus, we have that $v_1= iu_2$, $v_2 = iu_1$ form a positively oriented orthonormal frame of $N$, because $u_1,u_2,v_1,v_2$ induces the same orientation as $u_1,iu_1,u_2,iu_2$, the natural orientation of $L$ as a complex manifold.

    The curvature for the tangent and normal bundles are given by the $2$-forms
    $$\Omega_T = d g(\nabla u_2,u_1), \quad \Omega_N = d g(\nabla v_2,v_1),$$
    from where we conclude $e(N)=-e(T\Sigma)=-\chi(\Sigma)$. 
    \end{proof}
    
    \begin{thm} If $\Sigma$ is embedded as a section and does not have complex tangent spaces, then $-3|\tau| = 2e+2\chi$.
    \end{thm}
    \begin{proof} By Lemma \ref{lagrangian section}, we may assume there exists a point $\pp$ such that $\omega|_{T_\pp \Sigma} \neq 0$.

    Since the punctured surface $\Sigma':=\Sigma\setminus \{\pp\}$ is homotopically equivalent to a graph, every rank $2$ oriented real vector bundle is trivial over it. Thus, $T\Sigma'$ is a trivial bundle and we can consider a tangent unit vector field $u$ over $\Sigma'$. 
    
    The morphism of bundles
    $$\lambda: N|_{\Sigma'} \to \RR$$
    given by $(\xx,v) \mapsto \omega(u(\xx),v)$ has constant rank $1$. Indeed, if for some $\xx \in \Sigma'$  we have $\omega(u(\xx),v)=0$ for all $v \in N$, then $N_\xx = u(\xx)^\perp$ with respect to the Hermitian form. Since $T\Sigma$ and $N$ are orthogonal, that would imply that $T_\xx \Sigma = \CC u(\xx)$, contradicting the hypothesis of the theorem. 

    Therefore, we have a line bundle given by $\ker \lambda \subset N$. Inside $N$ we consider the line bundle $J \to \Sigma'$, orthogonal to $\ker \lambda$ with respect to the Riemannian metric $g$. The line bundle $J$ is trivial since for each $n \in J_\xx$ we have $\lambda(n)>0$ or $\lambda(n)<0$. Consider a section $n$ for $J \to \Sigma'$ such that $|n|=1$ and $\omega(u(\xx),n(\xx))>0$. By rotating $n$ by $90$ degrees in the counterclockwise direction, we obtain a section $v$ for $\ker \lambda$ with $|v| =1$.

    Thus, we have the unit vector fields $u,v$ with $\omega(u,v) = 0$. This means that $\langle u, v \rangle = 0$, that is, $u,v$ are orthogonal with respect to the Hermitian metric.

    Consider two unit vector fields $e_1,e_2$ defined in a neighborhood of $\pp$ with $e_1$ tangent to $\Sigma$, $e_2$ normal to $\Sigma$, and $\langle e_1,e_2 \rangle = 0$. These vector fields exist by the same argument we did for $u$ and $v$.

    Let $e_1'$ be the rotation of $e_1$ by $90$ degrees in the counter-clockwise direction at $T\Sigma$, and $e_2'$ the rotation of $e_2$ by $90$ degrees in the counter-clockwise direction at $N$. Thus, the local frames $e_1,e_1'$ and $e_2,e_2'$ are positively oriented in $T\Sigma$ and $N$.

    Consider a small closed disc $D$ centered at $\pp$ where the sections $e_1,e_1',e_2,e_2'$ are defined. We may assume that $\omega(e_1,e_1')$ does not vanishes over $D$. Consider as well a covering map $g:\RR \to \partial D$ with period~$1$. 

    From the map $\partial D \to \SP^1$ given by $\xx \to \left(g(u(\xx),e_1(\xx)),g(u(\xx),e_1'(\xx))\right)$, there exists a map $\theta: \RR \to \RR$ such that 
    $$u(t) = \cos(\theta(t)) e_1+\sin(\theta(t)) e_1',$$
    where we abbreviate $u(g(t))$ as $u(t)$.
    Similarly, we have the map $\phi:\RR \to \RR$ such that
    $$v(t) = \cos(\phi(t)) e_2+\sin(\phi(t)) e_2'.$$

    In the same style as the Poincaré-Hopf theorem, we have
    $$\chi= e(T\Sigma)=\frac{\theta(1)-\theta(0)}{2\pi},\quad  e(N) = \frac{\phi(1)-\phi(0)}{2\pi}.$$

    Now observe that,
    $$u(t) \wedge v(t) = \left(\cos\left(\theta(t)+\phi(t)\right) - i\omega (e_1(t),e_1'(t))\sin\left(\theta(t)+\phi(t)\right) \right) e_1 \wedge e_2$$
    and, therefore, the Euler number of $\wedge^{\!2} TL$ is 
    $e+\chi$ if $\omega(e_1,e_1')<0$, and $-(e+\chi)$ if $\omega(e_1,e_1')>0$.

    If $\tau\geq 0$, we can take $\pp$ for which $\omega(e_1,e_1')>0$, thus 
    $$\frac32\tau=-e-\chi.$$
    On the other hand, if $\tau<0$, we take $\pp$ for which $\omega(e_1,e_1')<0$ and we obtain
    $$\frac32\tau=e+\chi.$$

    Therefore, $-3|\tau| = 2e+2\chi$.

    \end{proof}

\end{document}